\documentclass[11pt]{article}

\usepackage{amsmath}
\usepackage{amsfonts}
\usepackage{amsthm}
\usepackage{thm-restate}
\usepackage{enumerate}
\usepackage{framed}
\usepackage{graphicx}
\usepackage{float}

\newtheorem{thm}{Theorem}

\newtheorem{prop}[thm]{Proposition}

\newtheorem{conj}[thm]{Conjecture}

\theoremstyle{definition}

\newcommand{\Z}{\mathbb{Z}}

\newcommand{\R}{\mathbb{R}}

\newcommand{\pO}{\mathcal{O}}

\newcommand{\RP}{\mathbb{RP}}

\newcommand{\es}{\emptyset}
\newcommand{\sm}{\setminus}
\newcommand{\nin}{\not\in}

\newcommand{\pow}{P\'{o}r and Wood}

\begin{document}

\nocite{*}

\title{A counterexample to a geometric Hales-Jewett type
        conjecture}

\author{Vytautas~Gruslys\footnote{%
Department of Pure Mathematics and Mathematical Statistics,
Centre for Mathematical Sciences,
University of Cambridge,
Wilberforce Road,
Cambridge,
CB3 0WB,
United Kingdom;
e-mail: {\tt v.gruslys@dpmms.cam.ac.uk}
}}

\maketitle

\begin{abstract}
    \pow{} conjectured that for all
    $k,l \ge 2$ there exists $n \ge 2$ with the following property:
    whenever $n$ points, no $l + 1$ of which are collinear, are chosen in the
    plane and each of them is assigned one of $k$ colours, then there must
    be a line (that is, a maximal set of collinear points) all of whose
    points have the same colour. The conjecture is easily seen to be true
    for $l = 2$ (by the pigeonhole principle) and in the case $k = 2$ it is
    an immediate corollary of the Motzkin-Rabin theorem. In this note we show
    that the conjecture is false for $k, l \ge 3$.
\end{abstract}


\section{Introduction}

Given a finite set $S$ in the plane, we will use the term \emph{line} to
denote any maximal set of collinear points of $S$. \pow{} posed the following
conjecture.

\begin{conj}[\pow{} \cite{por10}]
    \label{conj:main}
    For all integers $k \ge 1$ and $l \ge 2$, there is an integer $n$ such
    that for every finite set $S$ of size $|S| \ge n$ in the plane $\R^2$,
    if each point of $S$ is assigned one of $k$ colours, then

    \begin{itemize}
        \item $S$ contains $l+1$ collinear points, or
        \item $S$ contains a monochromatic line.
    \end{itemize}
\end{conj}

The motivation for this conjecture comes from the Hales-Jewett theorem. By a
\emph{combinatorial line} in the grid $[l]^n \subset \R^n$ (where $[l]$ stands
for the set $\{1, 2, \dotsc, l\}$) we mean a set of the form
$$
\{(x_1, \dotsc, x_n) \in [l]^n : x_i = x_j \text{ for all } i,j \in I\}
$$
for fixed $I \subset [n], I \neq \es$ and fixed $x_i, i \in [n] \sm I$.
Now the Hales-Jewett theorem can be stated as follows.

\begin{thm}[Hales and Jewett \cite{hales63}]
    \label{thm:hales-jewett}
    For all integers $k, l \ge 1$, there is an integer $n$ such that whenever
    each of the points in $[l]^n \subset \R^n$ is given one of $k$ colours,
    there is a monochromatic combinatorial line.
\end{thm}

Conjecture \ref{conj:main} is a natural geometric version of this theorem,
where the lines are not necessarily parallel to a fixed set of axes, and the
ambient set can be any set without many collinear points.

For $l=2$ the result is trivial: we may take $n = k+1$ and by the pigeonhole
principle there is a line containing two points of the same colour. The case
$k=2$ is a special case of the Motzkin-Rabin theorem that was proved in
\cite{motzkin67}.  In this paper we demonstrate by a counterexample that the
conjecture is false in the next smallest case $k = l = 3$, and hence it is
false whenever $k, l \ge 3$.

\begin{thm} \label{thm:main}
    For any $n \ge 2$, there is a set $S \subset \R^2$ of size $n$ satisfying:
    \begin{itemize}
        \item no four points of $S$ are collinear, and
        \item the points of $S$ can be coloured using three colours in such a
            way that no line is monochromatic.
    \end{itemize}
\end{thm}

\section{Proof of Theorem \ref{thm:main}} 

We start by noting that it is sufficient to find a set with the required
properties in the projective plane $\RP^2$. Indeed, given a finite set
$S \subset \RP^2$, one can choose a line $l \subset \RP^2$ that does not
meet $S$ and apply a projective transformation that sends $l$ to the line
at infinity. The image of $S$ under this transformation is contained in the
affine plane $\R^2$ while the collinearity relations of the original set
$S$ are preserved.

Our counterexample is a finite subset of the irreducible cubic curve
$y^2 = x^3 - x^2$.  More specifically, we will use a subset of the set of
its non-singular points
$\Gamma = \{(x,y) \in \R^2 : y^2 = x^3 - x^2, x \neq 0\} \cup
\{\pO\} \subset \RP^2$
where $\pO$ is a point at infinity that is contained in all lines parallel
to the $y$-axis and in the line at infinity. By the B\'ezout theorem, $\Gamma$
does not contain a set of four collinear points. Moreover, it is a well known
fact in algebraic geometry that $\Gamma$ forms an abelian group with the
property that distinct points $P, Q, R \in \Gamma$ are collinear if and only
if $P+Q+R=0$, and that $\Gamma$ is isomorphic to the circle group $\R/\Z$
(see \cite{husemoeller04}, p. 19--20).

In fact, any choice of an elliptic curve whose group is isomorphic to
$\R / \Z$ would do.  However, we choose this particular cubic curve (which
is not an elliptic curve as it contains a singular point $(0,0)$) because
it admits a simple explicit group isomorphism $\phi : \R/\Z \to \Gamma$,
given by
$$
    \phi(x) =
    \begin{cases}
        (\cot(\pi x)^2+1, \cot(\pi x)(\cot(\pi x)^2+1))
             &\text{ if } x \neq 0, \\
        \pO &\text{ if } x = 0.
    \end{cases}
$$
This enables us to give a self-contained proof of the theorem without
referring to any results from algebraic geometry. However, the reader
familiar with elliptic curves can skip the proof of the following
proposition.

\begin{prop} \label{prop:isom}
    Let $x, y$ and $z$ be distinct elements of $\R / \Z$. Then the points
    $\phi(x), \phi(y)$ and $\phi(z)$ are collinear if and only if $x+y+z=0$.
    Moreover, \mbox{$\phi : \R / \Z \to \Gamma$} is a well defined bijection.
\end{prop}

\begin{proof}
    The fact that $\phi$ is a well defined bijection follows from the basic
    properties of the cotangent function. To prove the equivalence of the
    geometric and algebraic relations, we will use the identity
    \begin{equation} \label{eq:cot}
        \cot(x+y) = \frac{\cot(x)\cot(y)-1}{\cot(x)+\cot(y)},
    \end{equation}
    which holds whenever $x+y, x, y$ are not multiples of $\pi$. Given a real
    number $r \nin \Z$, define $c_r = \cot(\pi r)$.

    If one of $x, y, z \in \R/\Z$ is $0$ (say, $x=0$) then $\phi(z)$
    is collinear with $\phi(x) = \pO$ and $\phi(y)$ if and only if $\phi(z)$
    is the reflection of $\phi(y)$ in the $x$-axis, that is, $z = -y$.
    Similarly, if two of the numbers (say, $x$ and $y$) sum to $0$, then
    the three points are collinear if and only if $\phi(z) = \pO$, that is,
    $z = 0$. Now we can assume that $x, y, z$ are all non-zero and that no
    two of them sum to $0$. Then the points $\phi(x), \phi(y)$ and $\phi(z)$
    are collinear if and only if
    $$
        \frac{c_z(c_z^2+1)-c_x(c_x^2+1)}{(c_z^2+1)-(c_x^2+1)}=
        \frac{c_z(c_z^2+1)-c_y(c_y^2+1)}{(c_z^2+1)-(c_y^2+1)},
    $$
    which after rearrangement becomes
    $$
        c_z = -\frac{c_x c_y - 1}{c_x+c_y}.
    $$
    Notice that $z = -x-y$ is a solution by (\ref{eq:cot}), and it is unique
    in $\R/\Z$ as $\cot$ is injective on $(0, \pi)$.
\end{proof}

Now we are ready to finish the proof of the theorem.

\begin{proof}[Proof of Theorem \ref{thm:main}]
    As noted before, it is enough to construct a set $S' \subset \RP^2$ with
    the two required properties, and take a projective transformation that
    maps $S'$ into $\R^2$.

    For the set $S'$ (see Fig.~\ref{fig:broad-scale} and \ref{fig:small-scale}) 
    we will take $S'=\{\phi(i/n) : i=0,\dotsc,n-1\}$. Notice
    that by Proposition \ref{prop:isom} there are no four collinear
    points in $S'$. Indeed, if $\phi(x), \phi(y), \phi(z)$ and $\phi(w)$ were
    distinct and collinear, then $z = -x-y = w$ in $\R / \Z$, giving a
    contradiction. Colour $\phi(i/n)$
    \begin{align*}
        \text{red } &\text{if } 0 \le i < \frac{n}{3}, \\
        \text{green } &\text{if } \frac{n}{3} \le i < \frac{2n}{3}, \\
        \text{blue } &\text{if } \frac{2n}{3} \le i < n.
    \end{align*}

    \begin{figure}[H] 
        \centering
        \includegraphics{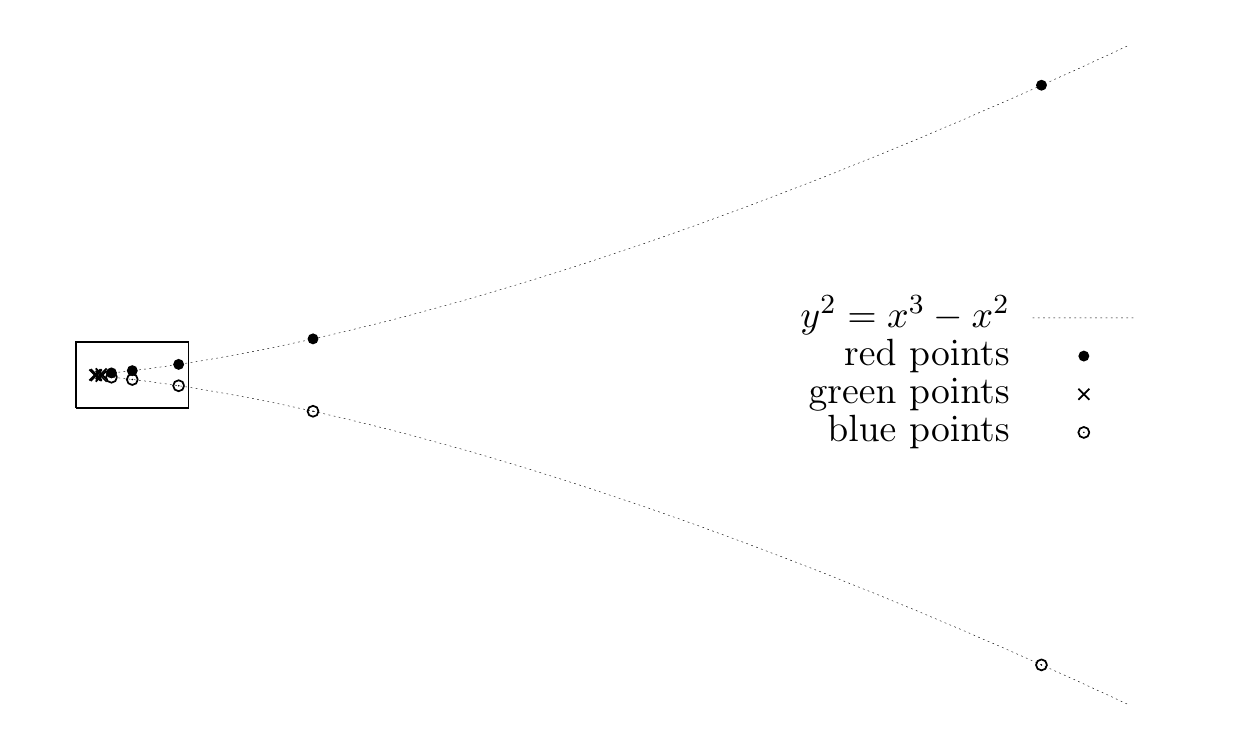}
        \caption{The set $S'$ with $n=16$. The sixteenth point is at infinity,
            and has red colour. The framed section is shown in smaller scale
            in Fig.~\ref{fig:small-scale}.}
        \label{fig:broad-scale}
    \end{figure}

    \begin{figure}[H] 
        \centering
        \includegraphics{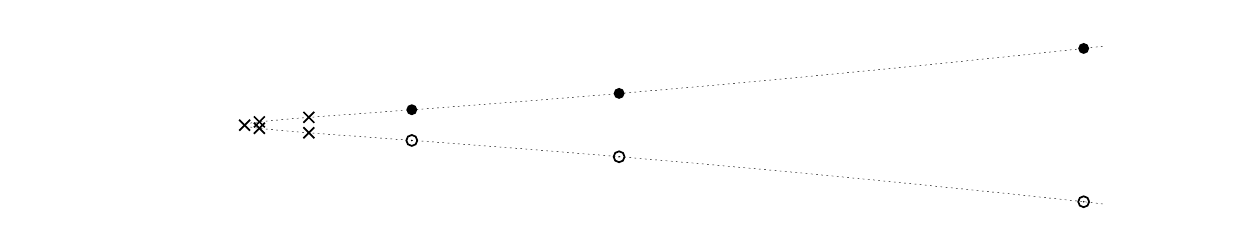}
        \caption{Part of the set $S'$ with $n=16$ in smaller scale.}
        \label{fig:small-scale}
    \end{figure}

    Suppose for contradiction that there is a monochromatic line $l$. It must
    pass through two distinct points $\phi(i/n)$ and $\phi(j/n)$, 
    $0 \le i,j < n$. There is an integer $0 \le k < n$ such that
    $k \equiv -i-j\,(\text{mod }n)$, possibly $k = i$ or $k = j$.
    Then $i/n + j/n + k/n = 0$ in $\R / \Z$, and so by Proposition
    \ref{prop:isom} either $\phi(i/n), \phi(j/n)$ and $\phi(k/n)$ are
    dictinct colinear points, or $\phi(k/n)$ coincides with one of the other
    two points. In either case $l$ passes through all of these points, and
    hence they have the same colour.

    Now write $i/n = x + \alpha\,, j/n = x + \beta$ and $k/n = x + \gamma$,
    where $x \in \{0, \frac{1}{3}, \frac{2}{3}\}$ and $\alpha, \beta,
    \gamma \in [0, \frac{1}{3})$. Considered as real numbers,
    $3x$ and $i/n + j/n + k/n = 3x + \alpha + \beta + \gamma$ are integers,
    so $\alpha + \beta + \gamma$ is also an integer. But $0 \le \alpha,
    \beta, \gamma < \frac{1}{3}$, so this is only possible if $\alpha =
    \beta = \gamma = 0$. In particular, $i/n = j/n$,
    contradicting the assumption that $\phi(i/n) \neq \phi(j/n)$.

    This finishes the proof.
\end{proof}

\bibliographystyle{elsarticle-num}

\begin{thebibliography}{9}

\bibitem{hales63}
    A.W. Hales and R.I. Jewett,
    Regularity and positional games,
    Trans. Amer. Math. Soc. 106 (1963),
    222--229.

\bibitem{husemoeller04}
    D. Husem\"{o}ller,
    Elliptic curves
    (2nd ed., Springer, 2004).

\bibitem{motzkin67}
    T.S. Motzkin,
    Nonmixed connecting lines,
    Notices Amer. Math. Soc. 14 (1967),
    837.

\bibitem{por10}
    A. P\'{o}r and D.R. Wood,
    On visibility and blockers,
    J. Computational Geometry 1 (2010),
    29--40.

\end{thebibliography}

\end{document}